\pdfoutput=1
\RequirePackage{ifpdf}
\ifpdf 
\documentclass[pdftex]{sigma}
\else
\documentclass{sigma}
\fi

\newcommand{\wedgebar}{\barwedge}
\newcommand{\stirl}[2]{\genfrac{\{}{\}}{0pt}{}{#1}{#2}}
\newcommand{\de}{\partial}
\newcommand{\Hom}{\operatorname{Hom}}
\newcommand{\Coder}{\operatorname{Coder}}
\newcommand{\N}{\mathbb{N}}
\newcommand{\Z}{\mathbb{Z}}
\newcommand{\Q}{\mathbb{Q}}

\numberwithin{equation}{section}

\newtheorem{Theorem}{Theorem}[section]
\newtheorem{Corollary}[Theorem]{Corollary}
\newtheorem{Conjecture}[Theorem]{Conjecture}
\newtheorem{Lemma}[Theorem]{Lemma}
 { \theoremstyle{definition}
\newtheorem{Definition}[Theorem]{Definition}

\newtheorem{Remark}[Theorem]{Remark} }

\begin{document}

\allowdisplaybreaks

\newcommand{\arXivNumber}{1509.09032}

\renewcommand{\PaperNumber}{053}

\FirstPageHeading

\ShortArticleName{Universal Lie Formulas for Higher Antibrackets}

\ArticleName{Universal Lie Formulas for Higher Antibrackets}

\Author{Marco MANETTI~$^\dag$ and Giulia RICCIARDI~$^{\ddag\S}$}

\AuthorNameForHeading{M.~Manetti and G.~Ricciardi}

\Address{$^\dag$~Dipartimento di Matematica ``Guido Castelnuovo'', Universit\`a degli studi di Roma La Sapienza,\\
\hphantom{$^\dag$}~P.~le Aldo Moro 5, I-00185 Roma, Italy}
\EmailD{\href{mailto:manetti@mat.uniroma1.it}{manetti@mat.uniroma1.it}}
\URLaddressD{\url{http://www1.mat.uniroma1.it/people/manetti/}}

\Address{$^\ddag$~Dipartimento di Fisica ``E.~Pancini'', Universit\`a degli studi di Napoli Federico II,\\
\hphantom{$^\ddag$}~Complesso Universitario di Monte Sant'Angelo, Via Cintia, I-80126 Napoli, Italy}

\Address{$^\S$~INFN, Sezione di Napoli, Complesso Universitario di Monte Sant'Angelo, \\
\hphantom{$^\S$}~Via Cintia, I-80126 Napoli, Italy}
\EmailD{\href{mailto:giulia.ricciardi@na.infn.it}{giulia.ricciardi@na.infn.it}}
\URLaddressD{\url{http://people.na.infn.it/~ricciard/}}

\ArticleDates{Received November 17, 2015, in f\/inal form May 31, 2016; Published online June 06, 2016}

\Abstract{We prove that the hierarchy of higher antibrackets (aka higher Koszul brackets, aka Koszul braces)
of a linear operator $\Delta$ on a commutative superalgebra can be def\/ined by some universal formulas
involving iterated Nijenhuis--Richardson brackets having as arguments $\Delta$ and the multiplication operators. As a byproduct,
we can immediately extend higher antibrackets to noncommutative algebras
in a way preserving the validity of genera\-li\-zed Jacobi identities.}

\Keywords{Lie superalgebras; higher brackets}
\Classification{17B60; 17B70}

\section{Introduction}\label{sec.introduction}

In particle physics, the fundamental interactions of the standard model, the electroweak and the QCD interactions, are described by non-abelian gauge theories. The presence of a gauge symmetry implies the appearance of unphysical degrees of freedom in the Lagrangian, which stand in the way of the usual quantization methods. Typically, the redundant degrees of freedom are removed through a gauge-f\/ixing procedure. Ghosts, i.e., f\/ields with unphysical statistics, are introduced to compensate for ef\/fects of the gauge degrees of freedom and preserve unitarity. The gauge-f\/ixed action retains a nilpotent, odd, global symmetry involving transformations of both f\/ields and ghosts, the Becchi--Rouet--Stora--Tyutin (BRST) symmetry \cite{Becchi:1974xu, Becchi:1974md, Becchi:1975nq, Tyutin:1975qk}. The BRST symmetry has played an important role in quantization, renormalization, unitarity, and other aspects of gauge theories. The Batalin--Vilkovisky formalism of antibrackets and antif\/ields \cite{Batalin:1981jr, Batalin:1984ss, Batalin:1985qj,getzler94} retains BRST symmetry as fundamental principle while dealing with very general gauge theories, included those with open or reducible gauge symmetry algebras. The antibracket formalism covers a broad spectrum of applications, ranging from supergravity to string and topological f\/ield theories. From the mathematical point of view, probably the most convenient approach to Batalin--Vilkovisky formalism is based upon a certain hierarchy of (super)symmetric multilinear maps introduced by Koszul \cite{koszul85} in the framework of dif\/ferential operators, Calabi--Yau manifolds and symplectic geometry.

Higher antibrackets, also known in literature as higher Koszul brackets or Koszul braces, were def\/ined in~\cite{akman, alfaro} as a~slight variation of Koszul construction and give a convenient generalization of the Batalin--Vilkovisky formalism working when the underlying algebra is not unitary and when the square-zero operator fails to be a dif\/ferential operator of second order. Several interesting mathematical properties of higher antibrackets have been studied in \cite{BDA,kravchenko2000} and more recently in~\cite{markl2, markl1}. In particular it is proved therein that higher antibrackets satisfy the generalized Jacobi identities, and then they provide a structure of strong homotopy Lie algebra, for every square-zero linear operator.

More recently, the formalism of higher antibrackets has been conveniently used in~\cite{FMpoisson} in the framework of Poisson geometry; here one of the key points was the existence of certain relations involving the linear, bilinear and trilinear antibrackets, the multiplication maps and the Nijenhuis--Richardson bracket on the space of multilinear operators. The initial motivation of this paper is to extend these relations to any order; more precisely we prove that every higher antibracket of a linear operator $\Delta$ on a commutative superalgebra can be def\/ined by some universal formulas involving iterated Nijenhuis--Richardson brackets, having as arguments $\Delta$ and the multiplication operators. In doing this we also get an explicit description of a universal gauge trivialization (Theorem~\ref{thm.gaugetriviality}) and of some representations of the Lie algebra of vector f\/ields of the line, which we think interesting in its own right (Section~\ref{sec.natural}). Moreover, as additional byproduct, some of the proved universal formulas for higher antibrackets,
when applied to operators in noncommutative algebras, give hierarchies of higher antibrackets satisfying generalized Jacobi identities, cf.~\cite{intrinsic}.

\section{Overview of the main results}\label{sec.overview}

Given a commutative superalgebra $A=A^{0}\oplus A^{1}$ over a f\/ield of characteristic 0 and a linear operator $f\colon A\to A$, the hierarchy of higher antibrackets of $f$ is the sequence of operators $\Phi_f^n\colon A^{\odot n}\to A$, $n>0$, def\/ined by the formulas \cite{koszul85,markl2}:
\begin{gather}
\Phi^1_f(a)= f(a),\nonumber\\
\Phi^2_f(a,b)= f(ab)-f(a)b-(-1)^{|a||b|}f(b)a,\nonumber\\
\Phi^3_f(a,b,c)= f(abc)-f(ab)c-(-1)^{|a|(|b|+|c|)}f(bc)a-(-1)^{|b||c|}f(ac)b\nonumber\\
\hphantom{\Phi^3_f(a,b,c)=}{} +f(a)bc+(-1)^{|a||b|}f(b)ac+(-1)^{|c|(|a|+|b|)} f(c)ab ,\nonumber\\
\cdots\cdots\cdots\cdots\cdots\cdots\cdots\cdots\cdots\cdots\cdots\cdots\cdots\cdots\cdots\cdots\cdots\cdots\cdots\cdots\cdots\nonumber\\
\Phi^n_f(a_1,\ldots,a_n)= \sum_{k=1}^n(-1)^{n-k}\sum_{\sigma\in S(k,n-k)}\epsilon(\sigma)
f(a_{\sigma(1)}\cdots a_{\sigma(k)})a_{\sigma(k+1)}\cdots a_{\sigma(n)} ,\label{equ.defiKoszul}
\end{gather}
where $A^{\odot n}$ is the symmetric $n$th power of $A$, where $|a|=0,1$ is the parity of a nonzero homogeneous element $a\in A$, and
$S(k,n-k)$ is the set of shuf\/f\/les of type $(k,n-k)$, i.e., the set of permutations $\sigma$ of $1,\ldots,n$
such that $\sigma(1)<\cdots<\sigma(k)$ and $\sigma(k+1)<\cdots<\sigma(n)$. The Koszul sign
$\epsilon(\sigma)$ is equal to $(-1)^{\alpha}$, where $\alpha$ is the number of pairs $(i,j)$ such that
$i<j$, $\sigma(i)>\sigma(j)$ and $|a_i||a_j|=1$. The same operators can be conveniently def\/ined by the recursive formulas \cite{akman}:
\begin{gather}
\Phi^1_f(a)=f(a),\nonumber\\
\Phi^{n+1}_f(a_1,\ldots,b,c) = \Phi^{n}_f(a_1,\ldots,bc)-\Phi^{n}_f(a_1,\ldots,b)c-(-1)^{|b||c|}\Phi^{n}_f(a_1,\ldots,c)b,\label{equ.defiakman}
\end{gather}
and dif\/fer from the def\/inition given in \cite{alfaro,BDA} by an extra sign factor appearing because in their set-up the operator $f$ acts on the right.
It is immediate to observe that every higher antibracket is a supersymmetric operator, i.e., we have
\begin{gather*} \Phi_f^n(\ldots,a_{i+1},a_i,\ldots)=(-1)^{|a_i| |a_{i+1}|}\Phi_f^n(\ldots,a_{i},a_{i+1},\ldots)\end{gather*}
for every index $0<i<n$ and every sequence $a_1,\ldots,a_n$ of homogeneous elements.
When $A$ has a unit $1\in A^0$, it is well known, and in any case easy to prove, that $\Phi^{n+1}_f=0$
if and only if $f(1)=0$ and $f$ is a dif\/ferential operator of order $\le n$~\cite{koszul85}.

The importance of higher antibrackets relies essentially on the fact that, up to a degree shifting, they satisfy the generalized Jacobi identities \cite{BDA,kravchenko2000}: this means that for every pair of operators
$f$, $g$ we have
\begin{gather*} \Phi^{n}_{[f,g]}=\sum_{i=1}^n\big[\Phi_f^i,\Phi_g^{n+1-i}\big]\end{gather*}
for every $n$, where $[-,-]$ denotes the Nijenhuis--Richardson bracket (see Section~\ref{sec.NR}) on the Lie superalgebra of symmetric operators
\begin{gather*}D(A)=\Hom^*\big(\oplus_{n>0}A^{\odot n},A\big) .\end{gather*}
The Lie superalgebra $D(A)$ contains as special even elements the multiplication operators
\begin{gather*} \mu_n\colon \ A^{\odot n+1}\to A,\qquad \mu_n(a_0,\ldots,a_n)=a_0\cdots a_n ,\qquad n\ge 0 .\end{gather*}
The associated adjoint operators
\begin{gather*} \rho_n\colon \ D(A)\to D(A),\qquad \rho_n(\omega)=[\mu_n,\omega]  \end{gather*}
are even derivations. A tedious but straightforward computation shows that:
\begin{gather*}
\bullet \ \Phi^2_f=-\rho_1(f)=-[\mu_1,f];\\
\bullet \ \Phi^3_f=\dfrac{1}{2}\big(\rho_1^2+\rho_2\big)(f)=\dfrac{1}{2}([\mu_1,[\mu_1,f]]+[\mu_2,f]);\\
\bullet \ \Phi^4_f=-\dfrac{1}{6}\big(\rho_1^3+3\rho_1\rho_2+6\rho_3\big)f=
-\dfrac{1}{6}([\mu_1,[\mu_1,[\mu_1,f]]]+3[\mu_1,[\mu_2,f]]+6[\mu_3,f]);\\
\bullet \ \Phi^5_f=\left(\dfrac{1}{30}\rho_1^4+\dfrac{1}{5}\rho_1^2\rho_2+\rho_1\rho_3+3\rho_4\right)f;\\
\bullet \ \Phi^6_f=-\dfrac{1}{3}\left(\dfrac{1}{90}\rho_1^5+\dfrac{1}{9}\rho_1^3\rho_2+\rho_1^2\rho_3+7\rho_1\rho_4+28\rho_5\right)f.
\end{gather*}
The above expressions for $\Phi^2_f$ and $\Phi^3_f$ were also pointed out and conveniently used in~\cite{FMpoisson} in the framework of Poisson geometry.

It is natural to ask whether similar formulas hold for every $n$: more precisely, we ask for a~sequence of noncommutative polynomials $P_n(\rho_1,\ldots,\rho_n)$ with rational coef\/f\/icients such that, for every~$f$, we have
\begin{gather}\label{equ.polinomigeneratori}
 \Phi_f^{n+1}=P_n(\rho_1,\ldots,\rho_n)f .
 \end{gather}
A f\/irst way to construct, at least in principle, the polynomials $P_n$, follows immediately from the following result.

\begin{Theorem}\label{thm.gaugetriviality}
In the above notation, for every linear operator $f\colon A\to A$ we have
\begin{gather*} \sum_{i=1}^{\infty}\Phi_f^i=\exp\left(-\sum_{n=1}^{\infty}K_n\rho_n\right)f, \end{gather*}
where $K_n$ is the sequence of rational numbers defined by the recursive equations
\begin{gather*} K_1=1,\qquad K_n=\frac{-2}{(n+2)(n-1)}\sum_{i=1}^{n-1} \stirl{n+1}{i}K_i ,\end{gather*}
where
\begin{gather*}\stirl{n+1}{i}=\frac{1}{ i! }\sum_{j=0}^i (-1)^{i-j}\binom{i}{j}j^{n+1}\end{gather*}
are the Stirling numbers of the second kind.
\end{Theorem}

The f\/irst 12 terms of $K_n$ are
\begin{gather*}K_1=1,\quad K_2=-\frac{1}{2},\quad K_3=\frac{1}{2},\quad K_4=-\frac{2}{3},\quad
K_5=\frac{11}{12},\quad K_6=-\frac{3}{4},\quad K_7=-\frac{11}{6}, \\
K_8=\frac{29}{4},\quad K_9=\frac{493}{12},\quad K_{10}=-\frac{2711}{6},
\quad K_{11}=-\frac{12406}{15},\quad K_{12}=\frac{2636317}{60},\end{gather*}
and we refer to Table~\ref{tab.tabellakoszul} for the values of $n! K_n$ for $n\le 16$. As kindly reported to the authors by one referee, this sequence appeared in certain computations about Hurwitz numbers contained in the paper~\cite{SZ} by Shadrin and Zvonkine, and especially in Conjecture~A.9, proved later by Aschenbrenner~\cite{Asche}; more details about that conjecture and its proof will be given in Remark~\ref{rem.proofA9}.

In practice it is more convenient to calculate the polynomials $P_n$ by using the following easy consequence of Theorem~\ref{thm.gaugetriviality}.

\begin{Corollary}\label{cor.recursive}
In the above notation, for every linear operator $f\colon A\to A$ we have
\begin{gather*} \Phi^1_f=f,\qquad \Phi^{n+1}_f=\frac{1}{n}\sum_{h=1}^n(-1)^{h}\rho_h\big(\Phi^{n-h+1}_f\big)=\frac{1}{n}\sum_{h=1}^n(-1)^{h}\big[\mu_h,\Phi^{n-h+1}_f\big].\end{gather*}
\end{Corollary}

The advantage of Corollary~\ref{cor.recursive} with respect to standard formulas \eqref{equ.defiKoszul} and
\eqref{equ.defiakman} is that it gives an easy way to def\/ine higher antibrackets
also for associative noncommutative superalgebras, providing that the symmetric operators $\mu_n\in D_n(A)$ are now def\/ined as
\begin{gather*} \mu_n(a_0,\ldots,a_n)=\frac{1}{(n+1)!}\sum_{\sigma\in \Sigma_{n+1}}\epsilon(\sigma)
a_{\sigma(0)}a_{\sigma(1)}\cdots a_{\sigma(n)} .\end{gather*}
In fact, also in this more general situation the brackets def\/ined in Corollary~\ref{cor.recursive}
satisfy the generalized Jacobi identities (see Remark~\ref{rem.noncommutative}),
while neither the symmetrization of~\eqref{equ.defiKoszul} nor the symmetrization of~\eqref{equ.defiakman} do. Extensions of higher antibrackets to noncommutative associative algebras were discussed in~\cite{alfaro,BM}, and a complete description has been given, with dif\/ferent approaches, in~\cite{derived,kapranovbrackets,Bering,intrinsic}.

The polynomials $P_n(\rho_1,\ldots,\rho_n)$ are not uniquely determined by equation~\eqref{equ.polinomigeneratori}. In fact, according to formula~\eqref{equ.bracketmuenne}, we have
\begin{gather}\label{equ.bracketrhoenne}
[\rho_n,\rho_m]=(n-m)\frac{(n+m+1)!}{(n+1)!(m+1)!}\rho_{n+m},\end{gather}
and therefore, by Poincar\'e--Birkhof\/f--Witt theorem, we can choose every $P_n$ as a linear combination with rational coef\/f\/icients of monomials of type
\begin{gather*} \rho_1^{i_1}\rho_2^{i_2}\cdots \rho_n^{i_n},\qquad \text{with}\qquad\sum_{h=1}^n hi_h=n .\end{gather*}

The following theorem, which is mathematically nontrivial, tells us in particular that, in the (super) commutative case, we can restrict to monomials such that $i_2+\cdots+i_n\le 1$.

\begin{Theorem}\label{thm.standardform}
In the notation above, for every integer $n>0$, there exists an unique sequence
$c_1^n,\ldots,c_n^n$ of rational numbers
such that, for every linear operator $f\colon A\to A$, we have
\begin{gather*}\Phi^{n+1}_f=\big(c_1^n\rho_1^n+c^n_2\rho_1^{n-2}\rho_2+c^n_3\rho_1^{n-3}\rho_3+\cdots+c^n_n\rho_n\big)f.\end{gather*}
\end{Theorem}

In Table~\ref{tab.coefficienti} we list the values of $(-1)^{n}n! c^n_i$ for every $n\le 10$:
as we shall see later, for a~f\/ixed~$n$, the $n$ coef\/f\/icients~$c_i^n$ can be computed as the solution of a suitable system of~$n+1$ linear equations. Moreover, some numerical computations suggest the validity of the following conjecture (verif\/ied for $n\le 12$) which we are not able to prove at the moment in full generality.

\begin{Conjecture} For every $n\ge 2$ the coefficients $c^n_i$ of Theorem~{\rm \ref{thm.standardform}} are given by the formula
\begin{gather*}
c_i^n=\frac{\displaystyle (-1)^{n} \prod_{j=2}^{i}\dfrac{n(n-1)-(j-1)(j-2)}{2}}
{\displaystyle \sum_{h=2}^n h\left(\prod_{j=2}^{h}\dfrac{n(n-1)-(j-1)(j-2)}{2}\right)\left(\prod_{j=h}^{n-1}\dfrac{(1-j)(j+2)}{2}\right)} ,
\end{gather*}
where every empty product is intended to be equal to $1$. Moreover $(-1)^{n}c^n_i>0$ for every $n\ge i\ge 1$.
\end{Conjecture}

\begin{table}\caption{Values of $x^n_i=(-1)^{n}n! c^n_i$ for $n\le 10$, see Theorem~\ref{thm.standardform}.}\label{tab.coefficienti}
\renewcommand\arraystretch{1.2}
$$ \begin{array}{c|cccccccccc}
n&2&3&4&5&6&7&8&9&10\\
\hline
\vphantom{\begin{matrix}1\\ 1\end{matrix}}x_1^n
&1&1&\dfrac{4}{5}&\dfrac{4}{9}&\dfrac{4}{21}&\dfrac{1}{15}&\dfrac{8}{405}&\dfrac{8}{1575}&\dfrac{4}{3465}\\
\vphantom{\begin{matrix}1\\ 1\end{matrix}}
x_2^n&1&3&\dfrac{24}{5}&\dfrac{40}{9}&\dfrac{20}{7}&\dfrac{7}{5}&\dfrac{224}{405}&\dfrac{32}{175}&\dfrac{4}{77}\\
\vphantom{\begin{matrix}1\\ 1\end{matrix}}
x_3^n&&6&24&40&40&28&\dfrac{224}{15}&\dfrac{32}{5}&\dfrac{16}{7}\\
\vphantom{\begin{matrix}1\\ 1\end{matrix}}
x_4^n&&&72&280&480&504&\dfrac{1120}{3}&\dfrac{1056}{5}&96\\
\vphantom{\begin{matrix}1\\ 1\end{matrix}}
x_5^n&&&&1120&4320&7560&\dfrac{24640}{3}&6336&3744\\
x_6^n&&&&&21600&83160&147840&164736&131040\\
x_7^n&&&&&&498960&1921920&3459456&3931200\\
x_8^n&&&&&&&13453440&51891840&94348800\\
x_9^n&&&&&&&&415134720&1603929600\\
x_{10}^n&&&&&&&&&14435366400
\end{array}$$
\renewcommand\arraystretch{1}
\end{table}

\section{The Nijenhuis--Richardson bracket}\label{sec.NR}

Unless otherwise specif\/ied every (super)vector space, symmetric product, linear map etc.\ are considered over the base f\/ield $\Q$ of rational numbers, although everything works in the same way over any f\/ield of characteristic~0.
Given a super vector space $V=V^{0}\oplus V^{1}$, we denote by~$V^{\odot n}$ its symmetric powers and by
$D_{n}(V)=\Hom^*(V^{\odot n+1},V)$ the super vector space of multilinear supersymmetric maps on $n+1$ variables
\begin{gather*} f\colon \ \underbrace{V\times\cdots\times V}_{n+1}\to V. \end{gather*}
Recall that supersymmetry means that
\begin{gather*} f(v_0,\ldots,v_i,v_{i+1},\ldots,v_n)=(-1)^{|v_i||v_{i+1}|}f(v_0,\ldots,v_{i+1},v_{i},\ldots,v_n),\end{gather*}
where $|v|=0,1$ denotes the parity of a homogeneous element $v$.

Given $f\in D_n(V)$ and $g\in D_m(V)$, $n,m\ge 0$, their Nijenhuis--Richardson bracket \cite{NijRich67} is def\/ined as
\begin{gather*} [f,g]=f\wedgebar g-(-1)^{|f||g|}g\wedgebar f \in D_{n+m}(V),\end{gather*}
where
\begin{gather*} f\wedgebar g
(v_{0},\ldots, v_{n+m})= \sum_{\sigma\in S(m+1,n)}\epsilon(\sigma)
f\big(g(v_{\sigma(0)},\ldots, v_{\sigma(m)}), v_{\sigma(m+1)},\ldots,
v_{\sigma(m+n)}\big),\end{gather*}
and $S(m+1,n)$ is the set of shuf\/f\/les of type $(m+1,n)$, i.e., the set of permutations $\sigma$ of $0,\ldots,n+m$
such that $\sigma(0)<\cdots<\sigma(m)$ and $\sigma(m+1)<\cdots<\sigma(m+n)$. The Koszul sign~$\epsilon(\sigma)$ is equal to $(-1)^{\alpha}$, where $\alpha$ is the number of pairs $(i,j)$ such that $i<j$, $\sigma(i)>\sigma(j)$ and $|v_i||v_j|=1$.

Observe that when $f\in D_0(V)$ the product $f\wedgebar g$ is the same as the composition product $f\circ g$ and therefore
the Nijenhuis--Richardson bracket reduces to the usual super commutator on~$D_0(V)$. We denote by $D(V)=\prod\limits_{n\ge 0}D_n(V)$; the Nijenhuis--Richardson brackets induces on~$D(V)$ a structure of Lie superalgebra:
\begin{gather*} \left[\sum_{i=0}^{\infty}f_i ,\sum_{j= 0}^{\infty}g_j\right]=\sum_{n= 0}^{\infty} \sum_{i=0}^n [f_i,g_{n-i}],\qquad f_i,g_i\in D_i(V) .\end{gather*}

\begin{Remark}\label{rem.coderivationNR}
Denoting by $\overline{S^c}(V)=\oplus_{n\ge 1} V^{\odot n}$ the reduced symmetric coalgebra generated by~$V$,
the composition with the natural projection $p\colon \overline{S^c}(V)\to V$ gives a morphism of super vector spaces
\begin{gather*} P\colon \ \Hom^*\big(\overline{S^c}(V),\overline{S^c}(V)\big)\to \Hom^*\big(\overline{S^c}(V),V\big)=D(V),\end{gather*}
and it is well known \cite{LadaMarkl} that $P$ induces an isomorphism of Lie superalgebras
\begin{gather*} P\colon \ \Coder^*\big(\overline{S^c}(V)\big)\xrightarrow{\simeq} D(V),\end{gather*}
where the bracket on $\Coder^*(\overline{S^c}(V))$ is the super commutator.
\end{Remark}

Assume now that $A$ is a commutative superalgebra, then there exists a distinguished sequence $\mu_n\in D_n(A)$, $n\ge 0$, corresponding to multiplication in~$A$:
\begin{gather*} \mu_n(a_0,\ldots,a_n)=a_0a_1\cdots a_n .\end{gather*}

We have
\begin{gather}\label{equ.bracketmuenne}
\mu_n\wedgebar \mu_m=\dbinom{n+m+1}{m+1}\mu_{n+m},\qquad
[\mu_n,\mu_m]=(n-m)\dfrac{(n+m+1)!}{(n+1)!(m+1)!}
\mu_{n+m} .\end{gather}
In fact, the binomial coef\/f\/icient in the f\/irst formula is equal to the cardinality of the set of shuf\/f\/les involved in the formula for the product $\wedgebar$.

\section{Koszul numbers}

Let's recall the notion of iterative exponential and iterative logarithm~\cite{ecalle}. Given a formal power series $a(t)\in\Q[[t]]$ with rational coef\/f\/icients such that $a(0)=a'(0)=0$, the following derivation is well def\/ined
\begin{gather*}a(t)\dfrac{d~}{dt}\colon \ \Q[[t]]\to \Q[[t]] ,\end{gather*}
as well as its exponential
\begin{gather*} \exp\left(a(t)\dfrac{d~}{dt}\right)\colon\ \Q[[t]]\to \Q[[t]] ,\end{gather*}
which is an isomorphism of rings. The power series
\begin{gather*}
\operatorname{itexp}(a(t))=\exp\left(a(t)\dfrac{d~}{dt}\right)(t)=t+a(t)+\frac{1}{2}a(t)a'(t)
+\frac{1}{6}(a(t)^2a''(t)+a(t)a'(t)^2)+\cdots
\end{gather*}
is called the \emph{iterative exponential} of $a(t)$. Conversely, for every $g(t)\in \Q[[t]]$ such that $g(0)=0$ and $g'(0)=1$ there exists a unique formal power series $a(t)=\operatorname{itlog}(g(t))$, called the \emph{iterative logarithm}, such that
\begin{gather*}
a(0)=a'(0)=0,\qquad g(t)=\operatorname{itexp}(a(t)) .
\end{gather*}

In this paper we are interested to a special sequence of rational numbers $K_i\in \Q$ which will appear in the natural description of the inf\/initesimal generator of the Koszul hierarchy: with a~certain lack of imagination we shall refer to this sequence as ``Koszul numbers''.

\begin{Definition}\label{def.koszulnumbers}
By \emph{Koszul numbers} we mean the sequence of rational numbers $K_n$, $n\ge 1$, determined by the formula
\begin{gather*}\sum_{n=1}^{\infty} K_{n}\frac{t^{n+1}}{ (n+1)! }=\operatorname{itlog}\big(e^t-1\big) ,\end{gather*}
or equivalently by the equation
\begin{gather*} \exp\left(\sum_{n=1}^{\infty}K_n\dfrac{t^{n+1}}{ (n+1)! }\dfrac{d~}{ dt }\right)(t)=e^t-1 .\end{gather*}
\end{Definition}

\begin{Remark}\label{rem.proofA9}
As already mentioned in the introduction, the f\/irst 12 Koszul numbers
\begin{gather*} K_1=1, \quad K_2=-\frac{1}{2}, \quad K_3=\frac{1}{2}, \quad K_4=-\frac{2}{3},\quad
K_5=\frac{11}{12}, \quad K_6=-\frac{3}{4}, \quad K_7=-\frac{11}{6},\\
K_8=\frac{29}{4}, \quad K_9=\frac{493}{12}, \quad K_{10}=-\frac{2711}{6}, \quad K_{11}=-\frac{12406}{15}, \quad K_{12}=\frac{2636317}{60},\quad \ldots\end{gather*}
appear in certain computations of Hurwitz numbers by Shadrin and Zvonkine~\cite{SZ}, in particular in Conjecture~A.9 of~\cite{SZ}, cf.\ also~\cite{sloane}. This conjecture was subsequently proved by Aschenbrenner~\cite{Asche}, who also f\/ixed a couple of misprints in the original formula of Shadrin and Zvonkine. Here we brief\/ly recall the conjecture together with the translation of the Aschenbrenner's proof in our setup, which avoids the use of inf\/inite matrices and combinatorial properties of the Stirling numbers.

The starting point is the following sequence of formal power series $a_d(z)\in \Q[[z]]$:
\begin{gather*} a_d(z)=\sum_{b=0}^{d}\frac{(-1)^{d-b}}{d!}\binom{d}{b}\frac{1}{1-(b+1)z}=
\sum_{n\ge 0}\sum_{b=0}^{d}\frac{(-1)^{d-b}}{d!}\binom{d}{b}(b+1)^nz^n,\qquad d\ge 0 .\end{gather*}
In order to show that the multiplicity of $a_d(z)$ is $d$ (cf.~\cite[Proposition~2.1]{SZ}), it is convenient to consider the isomorphism of $\Q$-vector spaces
\begin{gather*} \psi\colon \ \Q[[z]]\to t\Q[[t]],\qquad \psi\left(\sum c_nz^n\right)=\sum c_n \frac{t^{n+1}}{ (n+1)! },\end{gather*}
and prove that for every $d$ we have
\begin{gather}\label{equ.psiadi}
\psi(a_d(z))=\frac{(e^t-1)^{d+1}}{(d+1)!} .\end{gather}
In fact
\begin{gather*} \psi(a_d(z))=\sum_{n\ge 0}\sum_{b=0}^{d}\frac{(-1)^{d-b}}{d!}\binom{d}{b}(b+1)^n\frac{t^{n+1}}{(n+1)!}=
\sum_{b=0}^{d}\frac{(-1)^{d-b}}{d!}\binom{d}{b}\frac{e^{(b+1)t}-1}{b+1}, \end{gather*}
and therefore
\begin{gather*} \frac{d\psi(a_d(z))}{dt}=
\sum_{b=0}^{d}\frac{(-1)^{d-b}}{d!}\binom{d}{b}e^{(b+1)t}=\frac{e^t(e^t-1)^d}{d!}=
\frac{d~}{dt}\frac{(e^t-1)^{d+1}}{(d+1)!} .\end{gather*}
Thus we can def\/ine the rational numbers $a_{d,d+k}$, for $d,k\ge 0$, by setting
\begin{gather*} a_d(z)=\sum_{k\ge 0} a_{d,d+k}z^{d+k}=z^d+\sum_{k>0} a_{d,d+k}z^{d+k} .\end{gather*}
Consider now the ring $R=\Q[[t_0,t_1,\ldots]]$ of formal power series in the pairwise distinct indeterminates $t_n$, $n\in \N$. That ring
has the complete decreasing f\/iltration
\begin{gather*} F^pR=\left\{ \sum_{p\le i_0+2i_1+3i_2+\cdots<+\infty} a_{i_0,i_1,i_2,\ldots} t_0^{i_0}t_1^{i_1}t_2^{i_2}\cdots\right\},\end{gather*}
and therefore there exists a (unique) morphism of $\Q$-algebras $L\colon R\to R$ such that
\begin{gather*} L(t_d)=\sum_{k\ge 0}a_{d,d+k}t_{d+k}=t_d+\sum_{k>0}a_{d,d+k}t_{d+k}\end{gather*}
for every $d\ge 0$. Then the Conjecture~A.9 of \cite{SZ} is expressed by the equality:
\begin{gather}\label{equ.congA9}
L=\exp\left(\sum_{n>0}K_n\sum_{k\ge 0}\binom{k+n+1}{n+1}t_{n+k}\frac{\partial~}{\partial t_k}\right) .\end{gather}
In order to prove \eqref{equ.congA9} notice that the vector subspace $T\subset R$ of power series of type $\sum\limits_{n=0}^{\infty} a_nt_n$ is isomorphic to the maximal ideal of $\Q[[t]]$ via the continuous linear isomorphism of complete $\Q$-vector spaces
\begin{gather*} \phi\colon \ T\to t\Q[[t]],\qquad \phi(t_n)=\psi(z^n)=\frac{t^{n+1}}{(n+1)!} .\end{gather*}
For every $n\ge 0$ the operator $\phi^{-1}\circ \frac{t^{n+1}}{(n+1)!}\frac{d~}{dt}\circ \phi\colon T\to T$
is the restriction to $T$ of the derivation
\begin{gather*} r_n\colon \ R\to R,\qquad r_n=\sum_{k\ge 0} \binom{k+n+1}{n+1}t_{k+n}\frac{\partial~}{\partial t_k} .\end{gather*}
Therefore, by~\eqref{equ.psiadi} and the def\/inition of Koszul numbers, for every $d\ge 0$ we have
\begin{gather*}
\begin{split}
&\phi\left(\exp\left(\sum_{n>0}K_n\sum_{k\ge 0}\binom{k+n+1}{n+1}t_{n+k}
\frac{\partial~}{\partial t_k}\right)t_d\right) =
\exp\left(\sum_{n=1}^{\infty}K_n\dfrac{t^{n+1}}{ (n+1)! }\dfrac{d~}{ dt }\right)\frac{t^{d+1}}{(d+1)!}\\
& \hphantom{\phi\left(\exp\left(\sum_{n>0}K_n\sum_{k\ge 0}\binom{k+n+1}{n+1}t_{n+k}
\frac{\partial~}{\partial t_k}\right)t_d\right)}{}
 =\frac{(e^t-1)^{d+1}}{(d+1)!}=
\psi(a_d(z))=\phi(L(t_d)),
\end{split}
\end{gather*}
and this implies formula~\eqref{equ.congA9}.
\end{Remark}

\begin{Remark}\label{rem.magnus} Koszul numbers, and more generally iterative logarithms, may also be computed by using the pre-Lie Magnus expansion (see~\cite{pLDF} and references therein):
\begin{gather*} \sum_{n=1}^{\infty}K_n\dfrac{t^{n+1}}{ (n+1)! }\dfrac{d~}{ dt }=
\Omega\left(\sum_{n=1}^{\infty}\dfrac{t^{n+1}}{ (n+1)! }\dfrac{d~}{ dt }\right)\end{gather*}
in the pre-Lie algebra of formal vector f\/ields over the line.
\end{Remark}

\begin{Lemma}[Julia's equation]\label{lem.lemmachiave}
Given $a(t)\in \Q[[t]]$ such that $a(0)=a'(0)=0$ and denoting $g(t)=\operatorname{itexp}(a(t))$, we have $a(g(t))=a(t)g'(t)$.
\end{Lemma}

\begin{proof} The exponential $\exp\left(a(t)\frac{d~}{dt}\right)$ is a local isomorphism of complete local rings, therefore
\begin{gather*} \exp\left(a(t)\dfrac{d~}{dt}\right)(t^n)=g(t)^n\end{gather*}
for every $n$ and then
\begin{gather*} \exp\left(a(t)\dfrac{d~}{dt}\right)(f(t))=f(g(t))\end{gather*}
for every $f(t)\in \Q[[t]]$. In particular, we have
\begin{gather*}
a(g(t)) =\exp\left(a(t)\dfrac{d~}{dt}\right)(a(t))
=\exp\left(a(t)\dfrac{d~}{dt}\right)\left(a(t)\dfrac{d~}{dt}\right)(t)\\
\hphantom{a(g(t))}{} =\left(a(t)\dfrac{d~}{dt}\right)\exp\left(a(t)\dfrac{d~}{dt}\right)(t)
=\left(a(t)\dfrac{d~}{dt}\right)(g(t))\\
\hphantom{a(g(t))}{} =a(t)g'(t) . \tag*{\qed} \end{gather*}
\renewcommand{\qed}{}
\end{proof}

Julia's equation shows a clear relation between Koszul numbers and the sequence~A180609~\cite{hanna} in the On-Line Encyclopedia of Integer Sequences, def\/ined by the equation
\begin{gather*} A(t)=A(e^t-1)\frac{1-e^{-t}}{t},\qquad A(t)=\sum_{n=1}^{\infty} \frac{a_n t^n}{ n! (n+1)! },\qquad
a_1=1 .\end{gather*}

\begin{table}\caption{The f\/irst 16 terms of the sequence OEIS-A180609, $a_n=n! K_n$.}\label{tab.tabellakoszul}
$$\begin{array}{ll}
a_1=1&a_9=14908320\\
a_2=-1&a_{10}=-1639612800\\
a_3=3&a_{11}=-33013854720\\
a_4=-16&a_{12}=21046667685120\\
a_5=110&a_{13}=-549927873855360\\
a_6=-540&a_{14}=-637881314775344640\\
a_7=-9240&a_{15}=76198391578224115200\\
a_8=292320\qquad&a_{16}=41404329870413936025600
\end{array}$$
\end{table}

It is immediate to see that $K_n=\frac{a_n}{n!}$ for every $n>0$. In fact, setting $a(t)=tA(t)$ the equation $A(t)=A(e^t-1)\frac{1-e^{-t}}{t}$ becomes $a(e^t-1)=a(t)e^t$. In particular, the sequences~A134242 and~A134243~\cite{sloane} are respectively the numerators and the denominators of the sequence $a_n/n!$.

\begin{Lemma}[Aschenbrenner]\label{lem.aschelemma}
For every $n\ge 2$ we have
\begin{gather*}
K_n =\frac{-2}{(n+2)(n-1)}\sum_{i=1}^{n-1} \stirl{n+1}{i}K_i ,\\
K_{n} =\sum_{k=1}^n\frac{(-1)^{k+1}}{k}\sum_{1<n_1<\cdots<n_{k-1}<n_k=n+1}
\stirl{n_2}{n_1} \stirl{n_3}{n_2} \cdots \stirl{n_k}{n_{k-1}} ,
\end{gather*}
where $\stirl{n}{i}$ are the Stirling numbers of the second kind. In particular, $n!K_n$ is an integer for every~$n$.
\end{Lemma}

\begin{proof} This is proved in \cite[Section~7]{Asche} by using the theory of iterative logarithms. For reader's convenience we reproduce here the proof of the f\/irst equality as a consequence of Julia's equation.
Recall that the Stirling numbers of the second kind $\stirl{n}{i}$, $1\le i\le n$, may be def\/ined by the formulas
\begin{gather*}\stirl{n}{1}=\stirl{n}{n}=1,\qquad
\stirl{n}{i}=\stirl{n-1}{i-1}+i\stirl{n-1}{i},\qquad 1<i<n .\end{gather*}
In particular, for every $n\ge 2$ we have
\begin{gather*}
\stirl{n}{2}=2^{n-1}-1,\qquad \stirl{n}{n-1}=\binom{n}{2}.\end{gather*}
Let $f(t)\in \Q[[t]]$ be any formal power series and denote $g(t)=f(e^t-1)$. Then for every $n>0$ the $n$th derivative of $g$ is equal to
\begin{gather}\label{equ.equ.stirlingandderivata}
g^{(n)}(t)=\sum_{k=1}^n \stirl{n}{k} f^{(k)}\big(e^t-1\big)e^{kt} .
\end{gather}

The proof is done by induction on $n$. For $n=1$ formula~\eqref{equ.equ.stirlingandderivata}
is precisely the formula of derivation of composite functions. For general~$n$ we have
\begin{gather*}
g^{(n+1)}(t) =\sum_{k} \stirl{n}{k}\big(f^{(k+1)}\big(e^t-1\big)e^{(k+1)t}+if^{(k)}\big(e^t-1\big)e^{kt}\big),\\
g^{(n+1)}(t) =\sum_{k} \left(\stirl{n}{k-1}+k\stirl{n}{k}\right)f^{(k)}\big(e^t-1\big)e^{kt}.
 \end{gather*}
Notice that \eqref{equ.equ.stirlingandderivata} applied to
$f(t)=t^k$ gives
\begin{gather*}
g^{(n)}(0)=\stirl{n}{k}k! ,\quad
\frac{(e^t-1)^k}{k!}=\sum_{n\ge 0}\stirl{n}{k}\frac{t^n}{n!}=
\sum_{m\ge 0}\stirl{m+k}{k}\frac{t^{m+k}}{(m+k)!} .
\end{gather*}
Denoting $a(t)=K_1\frac{t^2}{2!}+K_2\frac{t^3}{3!}+\cdots$, according to
Def\/inition~\ref{def.koszulnumbers} and
Lemma~\ref{lem.lemmachiave} we have
$a(e^t-1)=a(t)e^{t}$;
taking the derivative we get
\begin{gather*} a'\big(e^t-1\big)e^t=a'(t)e^t+a(t)e^t,\qquad a(t)+a'(t)=a'\big(e^t-1\big),\end{gather*}
and then, by formula~\eqref{equ.equ.stirlingandderivata},
\begin{gather*} a^{(h)}(t)+a^{(h+1)}(t)=\sum_{i=1}^h \stirl{h}{i} a^{(i+1)}\big(e^t-1\big)e^{it} .\end{gather*}
Evaluating the above expression at $t=0$ we get
\begin{gather*} K_{h-1}+K_h=\sum_{i=1}^h \stirl{h}{i}K_i=K_h+\binom{h}{2}K_{h-1}+
\sum_{i=1}^{h-2} \stirl{h}{i}K_i, \\
 K_{h-1}=\frac{1}{1-\binom{h}{2}}\sum_{i=1}^{h-2} \stirl{h}{i}K_i=
\frac{-2}{(h+1)(h-2)}\sum_{i=1}^{h-2} \stirl{h}{i}K_i .\tag*{\qed}
\end{gather*}
\renewcommand{\qed}{}
\end{proof}

The symmetric coalgebra generated by the indeterminate $x$ over the f\/ield $\Q$ is the vector space $\Q[x]$ of polynomials with rational coef\/f\/icients, equipped with the coproduct
\begin{gather*} \Delta\colon \ \Q[x]\to \Q[x]\otimes_{\Q} \Q[x],\qquad \Delta(x^m)=\sum_{k=0}^{m}
\binom{m}{k}x^k\otimes x^{m-k}.\end{gather*}

\begin{Lemma}
Let $\de\colon \Q[x]\to \Q[x]$ be the usual derivation operator. Then the operators
\begin{gather*}d_n=\dfrac{x\partial^{n+1}}{ (n+1)! },\qquad n\ge -1,\end{gather*}
are coderivations in the coalgebra
$(\Q[x],\Delta)$ such that
\begin{gather}\label{equ.bracketdenne}
[d_n,d_m]=(n-m)\dfrac{(n+m+1)!}{(n+1)!(m+1)!}
d_{n+m} \end{gather}
for every $n,m\ge -1$.
\end{Lemma}

\begin{proof}
We have $d_n(x^m)=\binom{m}{n+1}x^{m-n}$, therefore
\begin{gather*}
\Delta d_n(x^m)= \sum_{k=0}^{m-n}\dbinom{m}{n+1}\binom{m-n}{k}x^k\otimes x^{m-n-k},\\
(d_n\otimes \operatorname{Id}+\operatorname{Id}\otimes d_n)\Delta(x^m)=
\sum_{s=0}^{m}\dbinom{m}{s}\binom{s}{n+1}x^{s-n}\otimes x^{m-s} \\
\hphantom{(d_n\otimes \operatorname{Id}+\operatorname{Id}\otimes d_n)\Delta(x^m)=}{} +
\sum_{k=0}^{m}\dbinom{m}{k}\binom{m-k}{n+1}x^{k}\otimes x^{m-k-n},
\end{gather*}
and the conclusion follows from the straightforward equality
\begin{gather*} \dbinom{m}{n+1}\binom{m-n}{k}=\dbinom{m}{k+n}\binom{k+n}{n+1}+
\dbinom{m}{k}\binom{m-k}{n+1} .\tag*{\qed}
\end{gather*}
\renewcommand{\qed}{}
\end{proof}

It is a trivial consequence of well known facts about symmetric coalgebras that every coderivation of $\Q[x]$ may be written as
\begin{gather*} \sum_{n=-1}^{\infty} a_n\dfrac{x\partial^{n+1}}{(n+1)!}\colon \ \Q[x]\to \Q[x]\end{gather*}
for a sequence $a_n\in \Q$, $n\ge -1$.

\begin{Lemma}\label{lem.dualita} Let $K_n$ be the sequence of Koszul numbers, then for every integer $h>0$ we have
\begin{gather*} \exp\left(\sum_{n=1}^{\infty}K_n \dfrac{x\partial^{n+1}}{(n+1)!}
\right)x^{h}=x+\text{higher order terms}.\end{gather*}
\end{Lemma}

\begin{proof}Consider the algebra $\Q[[t]]$ of formal power series as the algebraic dual of the coalgebra~$\Q[x]$, with the duality pairing given by
$\langle x^n,t^s\rangle=0$ if $n\not=s$ and $\langle x^n,t^n\rangle=n!$.
It is immediate to see that the dual of the coproduct $\Delta$ is precisely the Cauchy product, and the dual (transpose) of the coderivation $\frac{x\partial^{n+1}}{(n+1)!}$ is the derivation
$\frac{t^{n+1}\partial}{(n+1)!}$, since
\begin{gather*} \left\langle \dfrac{x\partial^{n+1}}{(n+1)!}x^{m+n},t^{m}\right\rangle=
m!\binom{m+n}{n+1}=\frac{m(m+n)!}{(n+1)!}=
\left\langle x^{m+n},\dfrac{t^{n+1}\partial}{(n+1)!}t^{m}\right\rangle
\end{gather*}
for every $n$, $m$.

Therefore $\exp\left(\sum\limits_{n=1}^{\infty}K_n\frac{t^{n+1}}{ (n+1)! }\frac{d~}{dt}\right)$ is the dual of $\exp\left(\sum\limits_{n=1}^{\infty}K_n \frac{x\partial^{n+1}}{(n+1)!}
\right)$ and then for every $h>0$
\begin{gather*} \left\langle \exp\left(\sum_{n=1}^{\infty}K_n \dfrac{x\partial^{n+1}}{(n+1)!}
\right)x^h, t\right\rangle=
\left\langle x^h, \exp\left(\sum_{n=1}^{\infty}K_n\dfrac{t^{n+1}}{ (n+1)! }\dfrac{d~}{dt}\right)t\right\rangle=1.\tag*{\qed}\end{gather*}
\renewcommand{\qed}{}
\end{proof}

\section{Proof of Theorem~\ref{thm.gaugetriviality}}

As in Section~\ref{sec.overview} let $A=A^{0}\oplus A^{1}$ be a commutative superalgebra; for every linear endomorphism $f\colon A\to A$ consider the hierarchy of higher antibrackets of $f$:
\begin{gather*}\Phi^n_f(a_1,\ldots,a_n)=
\sum_{k=1}^n(-1)^{n-k}\sum_{\sigma\in S(k,n-k)}\epsilon(\sigma)
f(a_{\sigma(1)}\cdots a_{\sigma(k)})a_{\sigma(k+1)}\cdots a_{\sigma(n)} .\end{gather*}

\begin{Lemma}[inversion formula] For every $f\in \Hom^*(A,A)$ and every $a_1,\ldots,a_n\in A$ we have
\begin{gather*}
f(a_1a_2\cdots a_n)=\sum_{k=1}^n \sum_{\sigma\in S(k,n-k)}\epsilon(\sigma)
\Phi^k_f(a_{\sigma(1)},\ldots, a_{\sigma(k)})a_{\sigma(k+1)}\cdots a_{\sigma(n)} .
\end{gather*}
\end{Lemma}

\begin{proof} The proof follows by a straightforward direct computation. We have
\begin{gather*}
\sum_{k=1}^n\sum_{\sigma\in S(k,n-k)}\epsilon(\sigma)
\Phi^k_f(a_{\sigma(1)},\ldots, a_{\sigma(k)})a_{\sigma(k+1)}\cdots a_{\sigma(n)} \\
\quad{} =\sum_{k=1}^n\sum_{\sigma\in S(k,n-k)}\sum_{h=1}^k(-1)^{k-h}\sum_{\tau\in S(h,k-h)}
\epsilon(\sigma)
f(a_{\sigma\tau(1)}\cdots a_{\sigma\tau(h)})\cdots a_{\sigma\tau(k)}a_{\sigma(k+1)}\cdots a_{\sigma(n)}\\
\quad{} =\sum_{h=1}^n\sum_{\sigma\in S(h,n-h)}\sum_{r=0}^{n-h}(-1)^{r}\sum_{\eta\in P_r(h,\sigma)}
\epsilon(\sigma)
f(a_{\sigma(1)}\cdots a_{\sigma(h)})a_{\eta\sigma(h+1)}\cdots a_{\eta\sigma(n)}\\
\quad {} =f(a_1\cdots a_n)+\sum_{h=1}^{n-1}\sum_{\sigma\in S(h,n-h)}\sum_{r=0}^{n-h}(-1)^{r}\sum_{\eta\in P_r(h,\sigma)}
\epsilon(\sigma)
f(a_{\sigma(1)}\cdots a_{\sigma(h)})a_{\eta\sigma(h+1)}\cdots a_{\eta\sigma(n)},
\end{gather*}
where $P_r(h,\sigma)$ is the set of permutations $\eta$ of $\{\sigma(h+1),\ldots,\sigma(n)\}$ such that
\begin{gather*} \eta(\sigma(h+1))<\eta(\sigma(h+2))<\cdots <\eta(\sigma(h+r)),\qquad 
\eta(\sigma(h+r+1))<\cdots <\eta(\sigma(n)).\end{gather*}
Since the cardinality of $P_r(h,\sigma)$ is $\binom{n-h}{r}$, for every $h<n$ and every shuf\/f\/le $\sigma\in S(h,n-h)$ we have
\begin{gather*}
\sum_{r=0}^{n-h}(-1)^{r}\sum_{\eta\in P_r(h,\sigma)}
\epsilon(\sigma)
f(a_{\sigma(1)}\cdots a_{\sigma(h)})a_{\eta\sigma(h+1)}\cdots a_{\eta\sigma(n)} \\
\qquad {} =\epsilon(\sigma) f(a_{\sigma(1)}\cdots a_{\sigma(h)})a_{\sigma(h+1)}\cdots a_{\sigma(n)}
\sum_{r=0}^{n-h}(-1)^{r}\dbinom{n-h}{r}=0.\tag*{\qed}
\end{gather*}
\renewcommand{\qed}{}
\end{proof}

\begin{Definition} For every $\alpha\in D(V)$ we will denote by $\widehat{\alpha}\in \Coder^*(\overline{S^c}(V))$ the (unique) coderivation such that
$P(\widehat{\alpha})=\alpha$.
\end{Definition}

The above def\/inition makes sense in view of the isomorphism of Lie superalgebras
\begin{gather*} P\colon \ \Coder^*\big(\overline{S^c}(A)\big)\xrightarrow{ \simeq } D(A)\end{gather*}
described in Remark~\ref{rem.coderivationNR}.

\begin{Lemma}\label{lem.exponentialstrange}
In the above setup, let $\mu=\sum\limits_{n=1}^{\infty}K_n\mu_n\in D(A)$. Then:
\begin{enumerate}\itemsep=0pt
\item[$1)$] the map $\exp(\widehat{\mu})\colon \overline{S^c}(A)\to \overline{S^c}(A)$ is an isomorphism of coalgebras and
\begin{gather*} P(\exp(\widehat{\mu}))=\sum_{n\ge 0}\mu_n\in D(A);\end{gather*}
\item[$2)$] for every $f\in\Hom^*(A,A)$ we have
\begin{gather*} \sum_{n>0}\Phi^n_f=\exp([-\mu,\cdot ])\Phi^1_f=\exp([-\mu,\cdot ])f.\end{gather*}
\end{enumerate}
\end{Lemma}

\begin{proof}
According to \cite{LadaMarkl,LadaStas} we have
\begin{gather*}\widehat{\mu}(a_0\odot\cdots\odot a_n)=\sum_{h=1}^n
K_h\sum_{\sigma\in S(h+1,n-h)} \epsilon(\sigma) a_{\sigma(0)}\cdots a_{\sigma(h)}\odot a_{\sigma(h+1)}\odot\cdots\odot a_{\sigma(n)},\end{gather*}
and then
\begin{gather*} p(\exp(\widehat{\mu})(a_0\odot\cdots\odot a_n))=q_n(K_1,\ldots,K_{n})a_0a_1\cdots a_n,\end{gather*}
where
\begin{gather*} q_n(K_1,\ldots,K_n)=\sum_{a_1+2a_2+\cdots+na_n=n}q_{a_1,\ldots,a_n}K_1^{a_1}\cdots K_n^{a_n}\end{gather*}
is a quasi-homogeneous polynomial of weight $n$ in $K_1,\ldots,K_n$, where $K_i$ has weight~$i$, and with coef\/f\/icients $q_{a_1,\ldots,a_n}\in \Q$ depending only on $a_1,\ldots,a_n$.

Thus, in order to prove that $q_n(K_1,\ldots,K_{n})=1$ for every~$n$, it is not restrictive to carry out the computations for the ``simplest'' algebra $A=\Q$.

In this case the coalgebra $\overline{S^c}(\Q)$ is isomorphic to the coalgebra $x\Q[x]$ equipped with the coproduct
\begin{gather*} \Delta\colon \ x\Q[x]\to x\Q[x]\otimes_{\Q} x\Q[x],\qquad \Delta(x^m)=\sum_{k=1}^{m-1}
\binom{m}{k}x^k\otimes x^{m-k},\end{gather*}
and the isomorphism is given by
\begin{gather*} a_1\odot\cdots \odot a_n\mapsto a_1a_2\cdots a_nx^n.\end{gather*}
Under this isomorphism, the projection $\overline{S^c}(\Q)\xrightarrow{ p }\Q$ corresponds to the evaluation in $0$ of the f\/irst derivative, the multiplication map $\mu_n\colon \overline{S^c}(\Q)\xrightarrow{\,}\Q$ corresponds to the linear map
\begin{gather*}m_n\colon \ \Q x^{n+1}\to \Q x,\qquad m_n\big(x^{n+1}\big)=x,\end{gather*}
and then the associated coderivation is
\begin{gather*} \widehat{m_n}(x^h)=\binom{h}{n+1}x^{h-n}=x\frac{\partial^{n+1}}{(n+1)!}(x^h),\qquad
\widehat{m_n}=\frac{x\partial^{n+1}}{(n+1)!} .
\end{gather*}
By Lemma~\ref{lem.dualita} we have $P\big(\exp\big(\sum K_n\widehat{m_n}\big)\big)=\sum m_n$ and this clearly implies $P(\exp(\widehat{\mu}))=\sum_n\mu_n$.

Given $f\colon A\to A$, according to inversion formula, for every $a_1,\ldots,a_n\in A$ we have
\begin{gather*}
p\exp (\widehat{\mu})\left(\sum_h \widehat{\Phi^h_f}(a_1\odot\cdots\odot a_n)\right) \\
\qquad{}=\sum_{h=1}^n\mu_{n-h}\left(\sum_{\sigma\in S(h,n-h)}\epsilon(\sigma)
\Phi^{h}_f(a_{\sigma(1)},\ldots,a_{\sigma(h)})
\odot a_{\sigma(h+1)}\odot\cdots\odot a_{\sigma(n)}\right)\\
\qquad{}=\sum_{h=1}^n\sum_{\sigma\in S(h,n-h)}\epsilon(\sigma)
\Phi^{h}_f(a_{\sigma(1)},\ldots, a_{\sigma(h)})
a_{\sigma(h+1)}\cdots a_{\sigma(n)}\\
\qquad{}=f(a_1a_2\cdots a_n)=fp\exp(\widehat{\mu})(a_1\odot\cdots\odot a_n)\\
\qquad{} =p\big(\widehat{f}\exp(\widehat{\mu})\big)(a_1\odot\cdots\odot a_n),
\end{gather*}
and then we have the equalities in $\Coder^*(\overline{S^c}(A))$:
\begin{gather*} \sum_n\widehat{\Phi^n_f}=\exp(\widehat{\mu})^{-1}\widehat{f}\exp(\widehat{\mu})=
\exp(-\widehat{\mu})\widehat{f}\exp(\widehat{\mu})=
\exp([-\widehat{\mu},\cdot ])\widehat{f} .\end{gather*}
It is now suf\/f\/icient to keep in mind that
$D(A)\xrightarrow{\widehat{\;\;\;}} \Coder^*(\overline{S^c}(A))$ is a Lie isomorphism.
\end{proof}

The proof of Theorem~\ref{thm.gaugetriviality} now follows from the fact that, by def\/inition the adjoint operator $[-\mu, \cdot ]$ is exactly $-\sum\limits_{n\ge 1}K_n\rho_n$.

\begin{Remark}If $F\colon \overline{S^c}(A)\to \overline{S^c}(A)$ denotes the unique isomorphism of coalgebras
such that $pF=\sum\limits_{n>0}\mu_n$, the above lemma shows in particular the well known equality $F^{-1}\widehat{f}F=
\sum_n\widehat{\Phi^n_f}$, cf.~\cite{kapranovbrackets,markl2,markl1}. An alternative, and more conceptual, proof of the f\/irst item of Lemma~\ref{lem.exponentialstrange} also follows from the results about pre-Lie exponential and pre-Lie Magnus expansion
discussed in \cite[Section~4]{pLDF}.
\end{Remark}

Theorem~\ref{thm.gaugetriviality} gives immediately the equalities, f\/irst proved in \cite{BDA},
\begin{gather*}
\Phi_{[f,g]}^{n}=\sum_{i=1}^n\big[\Phi^i_f,\Phi^{n-i+1}_g\big]
\end{gather*}
for every $n>0$ and every $f,g\colon A\to A$.
In fact, $\sum K_n\rho_n$ is an even derivation of $D(A)$, its exponential is an isomorphism of Lie superalgebras and then
\begin{gather*} \begin{split}
\sum_{i=1}^{\infty}\Phi_{[f,g]}^i&=\exp\left(-\sum_{n=1}^{\infty}K_n\rho_n\right)([f,g])\\
&=\left[\exp\left(-\sum_{n=1}^{\infty}K_n\rho_n\right)f,\exp\left(-\sum_{n=1}^{\infty}K_n\rho_n\right)g\right]\\
&=\left[\sum_{i=1}^{\infty}\Phi_{f}^i,\sum_{i=1}^{\infty}\Phi_{g}^i\right] .\end{split}\end{gather*}
In particular if $\Delta$ is a square-zero odd operator, then
$[\Delta,\Delta]=2\Delta^2=0$ and therefore
\begin{gather*} \sum_{i=1}^n\big[\Phi^i_\Delta,\Phi^{n-i+1}_\Delta\big]=0,\end{gather*}
i.e., $\Phi^n_\Delta$ are the Taylor coef\/f\/icients of an $L_{\infty}$ structure.

\begin{Corollary}\label{cor.recursivebis}
In the above notation, for every linear operator $f\colon A\to A$ we have
\begin{gather*} \Phi^1_f=f,\qquad \Phi^{n+1}_f=\frac{1}{n}\sum_{h=1}^n(-1)^{h}\big[\mu_h,\Phi^{n-h+1}_f\big].\end{gather*}
\end{Corollary}

\begin{proof} The higher antibrackets of the identity $\operatorname{Id}\colon A\to A$
are $\Phi^{n+1}_{\operatorname{Id}}=(-1)^n\mu_{n}$ and then
\begin{gather*} 0=\Phi^{n+1}_{[\operatorname{Id},f]}=\sum_{i=1}^{n+1}\big[\Phi^i_{\operatorname{Id}},\Phi^{n-i+2}_f\big]=
\big[\mu_0,\Phi^{n+1}_f\big]+\sum_{h=1}^{n}(-1)^h\big[\mu_h,\Phi^{n-h+1}_f\big] .\end{gather*}
The proof now follows from the trivial equality
$[\mu_0,\Phi^{n+1}_f]=-n\Phi^{n+1}_f$.
\end{proof}

\begin{Remark}\label{rem.noncommutative} There exists a natural way to extend the def\/inition of the operators $\mu_n\in D(A)$ to every noncommutative associative superalgebra $A$ by setting
\begin{gather*} \mu_n(a_0,\ldots,a_n)=\frac{1}{(n+1)!}\sum_{\sigma\in \Sigma_{n+1}}\epsilon(\sigma)
a_{\sigma(0)}a_{\sigma(1)}\cdots a_{\sigma(n)} ,\end{gather*}
and it is easy to see that the equality
\begin{gather*}[\mu_n,\mu_m]=(n-m)\dfrac{(n+m+1)!}{(n+1)!(m+1)!}
\mu_{n+m}\end{gather*}
still holds. The adjoint operators $\rho_n=[\mu_n,-]\colon D(A)\to D(A)$ are derivations and the map
\begin{gather*}\exp\bigg({-}\sum_{n\ge 1}K_n\rho_n\bigg)\colon \ D(A)\to D(A)\end{gather*}
is an isomorphism of Lie superalgebras. In particular, def\/ining the higher antibrackets by the formula
\begin{gather*} \sum_{i=1}^{\infty}\Phi_f^i=\exp\left(-\sum_{n=1}^{\infty}K_n\rho_n\right)f,\end{gather*}
the same argument as above shows the validity of the equalities
\begin{gather*} \Phi_{\operatorname{Id}}^{n+1}=(-1)^n\mu_n,\qquad \Phi_{[f,g]}^{n}=\sum_{i=1}^n\big[\Phi^i_f,\Phi^{n-i+1}_g\big],\end{gather*}
and the same proof of Corollary~\ref{cor.recursivebis} works in this case as well. Further properties of this noncommutative extension of Koszul brackets are studied in~\cite{intrinsic}, where it is proved in particular that they coincide with the brackets def\/ined by Bering and Bandiera, while they are dif\/ferent from the symmetrization of B\"{o}rjeson's brackets, cf.~\cite{markl1}.
\end{Remark}

\begin{Remark} For reference purpose, we point out that the sequence $K_n$ is uniquely determined by the formula
\begin{gather*} \sum_{i=1}^{\infty}\Phi_{\operatorname{Id}}^i=\exp\left(-\sum_{n=1}^{\infty}K_n\rho_n\right)\mu_0\end{gather*}
in the case $A=\Q$. In fact $\Phi^{n+1}_{\operatorname{Id}}=(-1)^{n}\mu_{n}$, $\rho_{n}(\mu_0)=n\mu_n$, and
therefore $((-1)^n+nK_n)\mu_n$ is the component in $D_n(\Q)$ of $\exp\left(-\sum\limits_{i=1}^{n-1}K_i\rho_i\right)\mu_0$.
\end{Remark}

\section{Natural higher brackets}
\label{sec.natural}

In order to prove the Theorem~\ref{thm.standardform} it is useful, and we believe interesting in itself, to consider the Koszul higher brackets
as special natural derived brackets: the notion of naturality has been discussed extensively in the paper \cite{markl1}; for our application we simply def\/ine a natural derived bracket of a linear endomorphism $f\colon A\to A$ as an element of type
\begin{gather*}\sum_{i+j=n}\psi_{ij} \mu_i\wedgebar (f\wedgebar\mu_j)\in D_n(A),\qquad \psi_{ij}\in \Q .\end{gather*}

We denote by
\begin{gather*}
\mathfrak{a}=\big\{\Psi\in \Hom_{\Q}(\Q[x],\Q[x])\,|\, \Psi(1)=0\big\}
=\left\{\sum_{n=0}^{\infty}p_n(x)\de^{n+1}\,|\, p_n(x)\in \Q[x]\right\}\end{gather*}
the associative algebra of linear endomorphisms of the $\Q$-vector space
$\Q[x]$ vanishing in~$1$, equipped with the composition product.
Every element $\Psi\in \mathfrak{a}$ is uniquely determined by its Taylor coef\/f\/icients,
i.e., by the inf\/inite double sequence $\psi_{ij}$, $i,j\ge 0$, def\/ined by the formulas:
\begin{gather}\label{equ.taylor}
\Psi(x^{i+1})=\sum_{j=0}^{\infty}\psi_{ij}\frac{x^j}{j!},\qquad i\ge 0,\qquad \psi_{ij}\in\Q .
\end{gather}

\begin{Definition} Let $\Psi\in \mathfrak{a}$ be a f\/ixed element.
For a commutative superalgebra $A$ over $\Q$ and a linear endomorphism $f\colon A\to A$,
the $\Psi$-bracket of $f$ is
\begin{gather*} \Psi_f\in D(A),\qquad
\Psi_f=\sum_{i,j\ge 0}\psi_{ij} \mu_j\wedgebar (f\wedgebar\mu_i),
\end{gather*}
where the $\psi_{ij}$'s are the Taylor coef\/f\/icients of $\Psi$ as in \eqref{equ.taylor}.
\end{Definition}

For every $n\ge i>0$, denote by $\Phi^{n,i}\in \mathfrak{a}$ the operator
\begin{gather*} \Phi^{n,i}(x^{i})=\frac{x^{n-i}}{(n-i)!} ,\qquad \Phi^{n,i}(x^s)=0\qquad \text{for}\quad s\not=i.\end{gather*}
Then for every $f\colon A\to A$ we have $\Phi^{n,i}_f=\mu_{n-i}\wedgebar (f\wedgebar\mu_{i-1})$:
\begin{gather*}\Phi^{n,i}_f(a_1,\ldots,a_{n})=\sum_{\sigma\in S(i,n-i)}\epsilon(\sigma)
f(a_{\sigma(1)}\cdots a_{\sigma(i)})a_{\sigma(i+1)}\cdots a_{\sigma(n)}.\end{gather*}
In particular, setting
\begin{gather*} \Phi^n=\sum_{i=1}^n(-1)^{n-i}\Phi^{n,i}\end{gather*}
we recover the def\/inition of higher Koszul brackets $\Phi^n_f$.
The formulas \eqref{equ.bracketrhoenne}, \eqref{equ.bracketmuenne} and \eqref{equ.bracketdenne} suggest the introduction of the Lie algebra
$\mathfrak{g}$ over the f\/ield $\Q$ of rational numbers with basis $l_0,l_1,l_2,\ldots$ and bracket
\begin{gather*}[l_n,l_m]=(n-m)\dfrac{(n+m+1)!}{(n+1)!(m+1)!}l_{n+m}.\end{gather*}
Notice that, setting $L_n=(n+1)! l_n$ we obtain the more familiar expression
\begin{gather*}[L_n,L_m]=(n-m)L_{n+m}.\end{gather*}
In particular, for every commutative superalgebra $A$ over $\Q$, the linear map
\begin{gather*} \mu_A\colon \ \mathfrak{g}\to D(A),\qquad \mu_A(l_n)=\mu_n\end{gather*}
is a morphism of Lie superalgebras: we denote by
\begin{gather*} \rho_A\colon \ \mathfrak{g}\to \Hom(D(A),D(A)),\qquad \rho_A(l_n)(\phi)=[\mu_n,\phi]\end{gather*}
the corresponding adjoint representation.

\begin{Lemma} The linear map $\rho\colon \mathfrak{g}\to \Hom(\mathfrak{a},\mathfrak{a})$:
\begin{gather*} \rho(l_n)(\Psi)=\left(\dfrac{x^n}{n!}-\dfrac{x^{n+1}\partial}{(n+1)!}\right)\circ \Psi-
\Psi\circ \dfrac{x\partial^{n+1}}{(n+1)!}\end{gather*}
is an injective morphism of Lie algebras.
\end{Lemma}

\begin{proof}We f\/irst prove that
\begin{gather*} \phi\colon \ \mathfrak{g}\to \Hom(\Q[x],\Q[x]),\qquad
\phi(l_n)= \dfrac{x\partial^{n+1}}{(n+1)!} ,\\
 \psi\colon \ \mathfrak{g}\to \Hom(\Q[x],\Q[x]),\qquad
\psi(l_n)=\dfrac{x^n}{n!}-\dfrac{x^{n+1}\partial}{(n+1)!} \end{gather*}
are morphisms of Lie algebras; it is easier to make the computation in the basis $L_n=(n+1)! l_n$,
\begin{gather*} \begin{split}
[\phi(L_n),\phi(L_m)]x^s&=\big[x\partial^{n+1},x\partial^{m+1}\big]x^s\\
&=x\partial^{n+1}\left(\prod_{i=0}^{m}(s-i)\right)x^{s-m}-x\partial^{m+1}\left(\prod_{i=0}^{n}(s-i)\right)x^{s-n}\\
&=(s-m)\prod_{i=0}^{m+n}(s-i)x^{s-m-n}-(s-n)\prod_{i=0}^{m+n}(s-i)x^{s-m-n}\\
&=(n-m)x\partial^{n+m+1}x^s=(n-m)\phi(L_{n+m})x^s.\end{split}\end{gather*}
Since $\psi(L_n)x^s=(n+1-s)x^{n+s}$ we have
\begin{gather*}\begin{split}
[\psi(L_n),\psi(L_m)]x^s&=\psi(L_n)(m+1-s)x^{m+s}-\psi(L_m)(n+1-s)x^{n+s}\\
&=((n+1-s-m)(m+1-s)-(m+1-n-s)(n+1-s))x^{n+m+s}\\
&=(n-m)(n+m+1-s)x^{n+m+s}=(n-m)\psi(L_{n+m})x^s.\end{split}\end{gather*}
Finally
\begin{align*}
[\rho(x),\rho(y)]\Psi&=\rho(x)\rho(y)\Psi-\rho(x)\rho(y)\Psi\\
&=\psi(x)\rho(y)\Psi-\rho(y)\Psi\phi(x)-
\psi(y)\rho(x)\Psi+\rho(x)\Psi\phi(y)\\
&=\psi(x) \psi(y)\Psi+\Psi\phi(y)\phi(x)
-\psi(y)\psi(x)\Psi-\Psi\phi(x)\phi(y)\\
&=[\psi(x),\psi(y)]\Psi-\Psi[\phi(x),\phi(y)]\\
&=\psi([x,y])\Psi-\Psi\phi([x,y])=\rho([x,y])\Psi .
\end{align*}
The injectivity is clear.
\end{proof}

\begin{Lemma} For every commutative superalgebra $A$ and every $f\in D_0(A)$ the map
\begin{gather*} \mathfrak{a}\to D(A),\qquad \Psi\mapsto \Psi_f\end{gather*}
is a morphism of $\mathfrak{g}$-modules. In other terms, for every $\Psi\in \mathfrak{a}$ and every $l\in \mathfrak{g}$ we have
\begin{gather}\label{equ.morfismogmoduli}
(\rho(l)\Psi)_f=\rho_A(l)(\Psi_f)=[\mu_A(l),\Psi_f] .
\end{gather}
\end{Lemma}

\begin{proof} By linearity is suf\/f\/icient to check \eqref{equ.morfismogmoduli} when $l=l_k$ and $\Psi=\Phi^{n,i}$. The proof follows from the following straightforward identities:
\begin{gather*}
\rho(l_k)\Phi^{n,i}=\left(\binom{n-i+k}{k}-\binom{n-i+k}{k+1}\right)\Phi^{n+k,i}-\binom{k+i}{k+1}\Phi^{n+k,i+k},\\
\big[\mu_k,\Phi^{n,i}_f\big]=\left(\binom{n-i+k}{k}-\binom{n-i+k}{k+1}\right)\Phi^{n+k,i}_f-\binom{k+i}{k+1}\Phi^{n+k,i+k}_f.\tag*{\qed}
\end{gather*}
\renewcommand{\qed}{}
\end{proof}

\begin{Theorem}\label{thm.basisforPhienne}
In the notation above, for every integer $k>0$ let's denote by $\rho_k=\rho(l_k)\colon \mathfrak{a}\to \mathfrak{a}$.
For every integer $n>0$, there exists an unique sequence $c^n_1,\ldots,c^n_n$ of rational numbers such that
\begin{gather*}\Phi^{n+1}=\big(c^n_1\rho_1^n+c^n_2\rho_1^{n-2}\rho_2+c^n_3\rho_1^{n-3}\rho_3+\cdots+c^n_n\rho_n\big)\Phi^1.\end{gather*}
\end{Theorem}

\begin{proof}Denote by $V^n\subset \mathfrak{a}$ the subspace generated by $\Phi^{n,1},\ldots,\Phi^{n,n}$; notice that $\dim V^n=n$ and
$\rho_k(V^n)\subset V^{n+k}$. We claim that, for every $n\ge 2$ the $n$ vectors
\begin{gather*}\rho_1^{n-2}\big(\Phi^{2,2}\big),\qquad \rho_1^a\rho_{n-a}\big(\Phi^{1,1}\big),\qquad 0\le a\le n-2,\end{gather*}
form a basis of $V^{n}$. This is clear for $n=2$ since $\rho_1(\Phi^{1,1})=\Phi^{2,1}-\Phi^{2,2}$.
By induction on $n$, it is suf\/f\/icient to prove that for every $n\ge 2$ the linear map
\begin{gather*} V^{n}\oplus V^1\to V^{n+1},\qquad (x,y)\mapsto \rho_1(x)+\rho_n(y)\end{gather*}
is an isomorphism. Consider f\/irst the case $n=2$: in terms of matrix multiplication we have
\begin{gather*} \big(\rho_1\big(\Phi^{2,2}\big),\rho_1\big(\Phi^{2,1}\big),\rho_2\big(\Phi^{1,1}\big)\big)=\big(\Phi^{3,3},\Phi^{3,2},\Phi^{3,1}\big)
\begin{pmatrix}-3&0&-1\\
1&-1&0\\
0&1&1\end{pmatrix}\end{gather*}
and the determinant of the matrix is $\not=0$. Assume now $n\ge 3$: since
\begin{gather*} \rho_n\big(\Phi^{1,1}\big)=\Phi^{n+1,1}-\Phi^{n+1,n+1}
,\qquad\rho_1\big(\Phi^{n,n-2}\big)=-\binom{n-1}{2}\Phi^{n+1,n-1},\end{gather*}
we can write
\begin{gather*} \big(\rho_1\big(\Phi^{n,n}\big),\ldots,\rho_1\big(\Phi^{n,1}\big),\rho_n\big(\Phi^{1,1}\big)\big)=
\big(\Phi^{n+1,n+1},\ldots,\Phi^{n+1,1}\big)
\begin{pmatrix}A&B\\
0&C\end{pmatrix},\end{gather*}
where $A\in M_{3,3}(\Z)$ and $C\in M_{n-2,n-2}(\Z)$ are lower triangular matrices with nonzero entries in the diagonal, and therefore the determinant of the block matrix is $\not=0$. Thus, for every $n$ there exists an unique sequence of $n+1$ rational numbers $c^n_1,\ldots,c^n_n,b_n$ such that
\begin{gather*}\Phi^{n+1}=\big(c^n_1\rho_1^n+c^n_2\rho_1^{n-2}\rho_2+\cdots+c^n_n\rho_n\big)\Phi^1+
b_n\rho_1^{n-1}\big(\Phi^{2,2}\big)\end{gather*}
and we have to show that $b_n=0$. To this end it is suf\/f\/icient to consider the representation
\begin{gather*} \mathfrak{a}\to D(\Q[x]),\quad \Psi\mapsto \Psi_\de,\end{gather*}
where, as usual $\de=\frac{d~}{dt}\in D_0(\Q[x])$. Since $\de$ is a derivation we have
$\Phi^{n+1}_\de=[\mu_n,\de]=0$ for every $n\ge 0$, and therefore it is suf\/f\/icient to prove that $\rho_1^{n}(\Phi^{2,2}_\de)\not=0$ for every $n$.
By construction $\Phi^{2,2}_\de(t^a,t^b)=(a+b)t^{a+b-1}$ and an easy induction on $n$ shows that for every $n\ge 3$ we have
\begin{gather*} \rho_1^{n-2}\big(\Phi^{2,2}_\de\big)\big(t^{a_1},t^{a_2},\ldots,t^{a_n}\big)=
\left[\prod_{i=3}^n\binom{i-1}{2}\right](a_1+\cdots+a_n)t^{\sum a_i-1}.\tag*{\qed}\end{gather*}
\renewcommand{\qed}{}
\end{proof}

\section*{Acknowledgments} We thank Bruno Vallette for useful comments and for bringing to our attention the pre-Lie Magnus expansion. We are indebted with the anonymous referees for several remarks and especially for letting us into the knowledge of the papers \cite{Asche,SZ}. M.M.~acknowledges partial support by Italian MIUR under PRIN project 2012KNL88Y ``Spazi di moduli e teoria di Lie''; G.R.~acknowledges partial support by Italian MIUR under PRIN project 2010YJ2NYW and INFN under specif\/ic initiative QNP.

\pdfbookmark[1]{References}{ref}
\LastPageEnding

\end{document}